\documentclass[12pt]{article}
\usepackage[a4paper,margin=25mm]{geometry}
\usepackage{lmodern}
\usepackage{xcolor,microtype,mathtools,amssymb,amsthm}
\usepackage{xurl}
\usepackage[colorlinks=true,linkcolor=blue!55!black,
  citecolor=blue!55!black,urlcolor=blue!55!black,
  bookmarksnumbered=true,bookmarksopen=true,bookmarksopenlevel=2,
  pdfpagemode=UseOutlines]{hyperref}

\numberwithin{equation}{section}

\newtheorem{theorem}{Theorem}[section]
\newtheorem{proposition}[theorem]{Proposition}
\newtheorem{lemma}[theorem]{Lemma}
\newtheorem{corollary}[theorem]{Corollary}
\newtheorem{conjecture}[theorem]{Conjecture}
\newtheorem{remark}[theorem]{Remark}

\newcommand{\N}{\mathbb N}
\newcommand{\Z}{\mathbb Z}
\newcommand{\psic}{\psi^{\circ}}
\newcommand{\e}{\mathrm e}

\title{\Large\bf Spectral asymptotics of sub-Riemannian Laplacians on compact Heisenberg manifolds}
\author{Sheng-Chen Mao \and Ye Zhang}
\date{}

\begin{document}

\maketitle

\noindent\textbf{Abstract}. Let \(N_M(\lambda)\) be the spectral counting function of the
sub-Laplacian on the compact Heisenberg manifold
\(M=\Gamma\backslash\mathbb H_d\), where $\Gamma$ is a lattice subgroup of the Heisenberg group $\mathbb H_d$. In 2016,  Strichartz  \cite[\textit{J. Geom. Anal.}]{Str16}  proved the Weyl law with
remainder \(R_M(\lambda)=N_M(\lambda)-A_d\operatorname{vol}(M)\lambda^{d+1} = O_M(\lambda^d\log\lambda)\), and conjectured the optimal
remainder to be \(O_M(\lambda^d)\). In this work, we establish a new upper bound and the first two-sided lower bounds
$$
R_M(\lambda)=O_M\!\left(\lambda^d(\log\lambda)^{2/3}\right),
\qquad
R_M(\lambda)=\Omega_{M,\pm}\!\left(\lambda^d\log\log\lambda\right).
$$
As a result, this implies that the sharp polynomial order is $d$, and disproves Strichartz's conjecture.

\medskip
\noindent\textbf{MSC2020.} Primary 35P20, 43A85; Secondary 11N37, 11L07, 58J50

\smallskip
\noindent\textbf{Keywords.} Sub-Riemannian Weyl law; Strichartz conjecture;
Heisenberg manifold; Sub-Laplacian

\section{Introduction}

The principal term in a Weyl law is determined by local geometry,
whereas the sharp remainder may depend on global spectral data.  For
the standard sub-Laplacian on a compact Heisenberg manifold, we prove
that the remainder at order \(\lambda^d\) contains an
unbounded arithmetic factor.  This factor is produced by the exact
representation spectrum and does not appear in the leading heat
asymptotics.

Let \(M=\Gamma\backslash\mathbb H_d\), \(d\geq1\), be a compact
Heisenberg manifold in the normalization of Folland \cite{Fol04}, and
let \(N_M(\lambda)\) count, with multiplicity, the eigenvalues of the
standard sub-Laplacian \(\mathcal L_0\) that do not exceed
\(\lambda\). In 2016, Strichartz \cite{Str16} proved
\begin{equation}\label{e01}
 N_M(\lambda)
 =A_d\operatorname{vol}(M)\lambda^{d+1}
 +O_M(\lambda^d\log\lambda),
\end{equation}
where \(A_d>0\) is explicit.  He also  conjectured the optimal
remainder to be \(O_M(\lambda^d)\) as follows.  
\begin{conjecture}[Strichartz]\label{strc}
For every compact Heisenberg manifold
\(M=\Gamma\backslash\mathbb H_d\),
\[
 R_M(\lambda):=N_M(\lambda)
  -A_d\operatorname{vol}(M)\lambda^{d+1} =O_M(\lambda^d).
\]
\end{conjecture}

Our main results disprove Conjecture~\ref{strc}.  More precisely, we
prove
\begin{equation}\label{e02}
 R_M(\lambda)
 =\Omega_\pm\!\left(\lambda^d\log\log\lambda\right),
 \qquad
 R_M(\lambda)
 =O_M\!\left(\lambda^d(\log\lambda)^{2/3}\right).
\end{equation}
The first assertion includes explicit lower bounds in the normalized
limsup and liminf inequalities.  Hence the remainder takes values of
both signs whose magnitude exceeds the order \(\lambda^d\) by an
unbounded factor.

The exact representation spectrum reduces \(R_M(\lambda)\), up to an
error of order \(O_M(\lambda^d)\), to one weighted sawtooth sum.  The
arithmetic estimates for this sum and their transfer to the spectral
remainder are described in Subsection~\ref{mthd}.

\subsection{Main results}
Let the center of \(\Gamma\) be
\[
 \Gamma\cap Z(\mathbb H_d)=\{(0,0,cn):n\in\mathbb Z\},
 \qquad c>0.
\]
After placing the lattice in the normal form of Folland
\cite[Proposition~2.1]{Fol04}, let \(L\) denote the associated positive
integer.  Equivalently, \(L=L_\Gamma\) is characterized by the spectral
multiplicity \(L|n|^d\) of the representation with nonzero central
frequency \(n\).  In Folland's normalization of the group and Haar
measure,
\begin{equation}\label{e03}
 \operatorname{vol}(M)=Lc^{d+1}.
\end{equation}
Hence the constants below depend only on the normalized quotient.
Define
\begin{equation}\label{e04}
 \kappa_M=\frac{L(c/\pi)^d}{(d-1)!}
 =\frac{\operatorname{vol}(M)}
 {c\pi^d(d-1)!}.
\end{equation}

The first theorem gives the two-sided lower bound for the spectral
remainder.  Its constant follows from the estimates for the integer
and half-integer shifts that occur in the Heisenberg spectrum.  Here
\(\gamma\) denotes Euler's constant.
\begin{theorem}
\label{som}
With \(\kappa_M\) as in \eqref{e04},
\begin{equation}\label{e05}
 \limsup_{\lambda\to\infty}
 \frac{R_M(\lambda)}
      {\lambda^d\log\log\lambda}
 \geq\e^\gamma\kappa_M, \qquad
 \liminf_{\lambda\to\infty}
 \frac{R_M(\lambda)}
      {\lambda^d\log\log\lambda}
 \leq-\e^\gamma\kappa_M.
\end{equation}
Consequently,
\(R_M(\lambda)=\Omega_\pm(\lambda^d\log\log\lambda)\).
\end{theorem}

The second theorem improves the logarithmic exponent in Strichartz's
estimate \eqref{e01}.
\begin{theorem}\label{sup}
For every compact Heisenberg manifold \(M=\Gamma\backslash\mathbb H_d\),
\begin{equation}\label{e06}
 N_M(\lambda)
 =A_d\operatorname{vol}(M)\lambda^{d+1}
 +O_M\!\left(\lambda^d(\log(2+\lambda))^{2/3}\right).
\end{equation}
\end{theorem}

About these two results, we have the following remark.
\begin{remark}
(1) Theorems \ref{som}  disproves Strichartz's Conjecture  \ref{strc}, which therefore  forbids the classical two-term expansion of Weyl law for $N_M$. Nevertheless, combining  it with Theorem  \ref{sup}, we can deduce that the sharp polynomial order for the Weyl remainder is exactly $d$.

(2) With \(\mu=\sqrt\lambda\) and \(Q=2d+2\), the two spectral theorems
read
\[
 R_M(\mu^2)
 =O_M\!\left(\mu^{Q-2}(\log\mu)^{2/3}\right),
 \qquad
 R_M(\mu^2)
 =\Omega_\pm\!\left(\mu^{Q-2}\log\log\mu\right).
\]
Thus the remainder occurs at the first lower homogeneous power below
the principal Weyl term \(\mu^Q\), multiplied by an unbounded factor.
\end{remark}

\subsection{Proof strategy}
\label{mthd}

Both the proofs of Theorems  \ref{som} and  \ref{sup} reduce the spectral remainder to a weighted sawtooth
sum.  We denote  the usual sawtooth function by
\[
 \psi(u)=u - \lfloor u\rfloor -\frac12
\]
 and given $T\ge2$ we define
\begin{equation}\label{e07}
 S_r(T)=\sum_{1\leq m\leq T}\frac1m
 \psi\left(\frac{T}{m+r}\right),
 \qquad r\geq0.
\end{equation}
The spectral reduction for \(R_M(\lambda)\) in \eqref{e14}, stated in
Proposition~\ref{pred}, reduces Theorems~\ref{sup} and~\ref{som} to
Theorems~\ref{tup} and~\ref{tom}, respectively.  In particular, the
coefficient and sign in the reduction \eqref{e14} give the constants
and signs in the spectral bounds \eqref{e05}.

For the upper bound, we reduce the rationally shifted sum \(S_r(T)\)
in \eqref{e07} to Walfisz's estimate on an arithmetic progression.  The
lower bound uses two complementary arguments.  For an integer \(n\),
the unshifted sum has the jump
\[
 S_0(n)-S_0(n^-)
 =-\sigma_{-1}(n)+\frac1{2n},
 \qquad
 \sigma_{-1}(n)=\sum_{a\mid n}\frac1a.
\]
Because \(\sigma_{-1}(n)=\sigma(n)/n\), Gronwall
\cite{Gro13} proved
\[
 \limsup_{n\to\infty}
 \frac{\sigma_{-1}(n)}{\log\log n}=\e^\gamma;
\]
see also Hardy and Wright \cite[Chapter~18]{HW08}.  A large jump by
itself gives no information about the sign of the midpoint.  To obtain
both signs with the stronger constant needed for the spectral problem,
we use P\'etermann's theorem \cite{Pet87,Pet88} for
\[
 E_{-1}(x)=\sum_{n\leq x}\sigma_{-1}(n)
 -\zeta(2)x+\frac12\log x.
\]
The identity for \(S_0(x)\) in \eqref{e35} gives
\(S_0(x)=-E_{-1}(x)+O(1)\).  Changes of index then transfer
P\'etermann's constants to every integer and half-integer shift.  Since
the spectral shift is \(d/2\), these estimates and the reduction for
\(R_M(\lambda)\) in \eqref{e14} yield the constants in
Theorem~\ref{som}.

For a general nonnegative rational shift, we prove a separate
two-sided estimate.  We assign the midpoint value to the sawtooth at
each integer and show that the range \(m>\sqrt T\) contributes
\(O_r(1)\).  If \(r=A/B\) and \(T=nQ\), the remaining summands are
periodic in \(n\) and have mean zero over a complete period.  Averaging
over \(1\leq n\leq Q^3\) controls the midpoint, while a common jump at
a suitable least common multiple gives both signs.  This argument
proves the constants in \eqref{e09}; P\'etermann's theorem gives the
stronger constants in \eqref{e09b} for the spectral shifts.

The direct argument applies to the shifted sawtooth sum \(S_r(T)\)
in \eqref{e07} and uses exact mean zero on the rational orbits
determined by \(r=A/B\).  Adhikari, Balasubramanian, and
Sankaranarayanan \cite{ABS} used an Erd\H{o}s--Shapiro construction to
prove a related one-sided \(\Omega_+\) estimate for the four-square
remainder.  The argument here treats every fixed nonnegative rational
shift and obtains both signs.  It is distinct from the amplification
method of Soundararajan \cite{Sou03} and the sectorial kernel
resonance method \cite{Lam26}, which use long trigonometric
polynomials; our lower bound follows from a simultaneous
discontinuity.

\subsection{Backgrounds}

We first compare the result with elliptic Weyl remainders, which
already encode global information in the
elliptic theory.  For a closed \(n\)-dimensional Riemannian manifold,
the local spectral function theorem of H\"ormander \cite{Hor68}
gives
\[
 N_g(\lambda)
 =\frac{\omega_n}{(2\pi)^n}\operatorname{vol}_g(X)\lambda^{n/2}
 +O_g\!\left(\lambda^{(n-1)/2}\right);
\]
Duistermaat and Guillemin \cite{DG75} related the wave
trace to closed geodesics and periodic bicharacteristics, while
Canzani and Galkowski \cite{CG23} obtained logarithmic
improvements by controlling beams associated with nearly periodic
geodesics.  Thus the principal Weyl coefficient is
local, but its sharp remainder is governed by global recurrence.

The closest elliptic comparison is the Laplace--Beltrami operator of
a left-invariant Riemannian metric on a compact Heisenberg manifold.
Petridis and Toth \cite{PT02} obtained an upper bound for
arithmetic metrics and
formulated the conjectural scale
\(\lambda^{3/4+\varepsilon}\); Chung, Petridis, and Toth
\cite{CPT03} treated arbitrary left-invariant metrics
through exponent pairs and exponential sum estimates.  Khosravi and Petridis
\cite{KP05}, Khosravi and Toth \cite{KT05},
Zhai \cite{Zha08}, and Nowak \cite{Now09,No09b} subsequently
established formulas for the mean square, power moments, distribution
results, and \(\Omega\) bounds.

The two problems share the same representation structure, but their
energy laws differ.  In the Riemannian problem the
non-Abelian levels contain a quadratic vertical term in the central
frequency, and continuous oscillation is analyzed by exponent pair
methods.  For the sub-Laplacian the vertical square is
absent: the energy is
\[
 \frac{\pi|n|}{c}\left(j+\frac d2\right),
\]
the Weyl exponent is \(Q/2=d+1\), and the relevant argument of the
sawtooth is \(T/(m+d/2)\).  The lower bound for \(R_M(\lambda)\) in
\eqref{e02} follows from simultaneous discontinuities, not from the
continuous oscillation that determines the Riemannian remainder.
Hence our theorem neither improves nor contradicts the Riemannian
conjecture; the two problems have different sources of spectral
oscillation on the same nilmanifold.

For the broader sub-Riemannian context, let
\((X,\mathcal H,g_{\mathcal H})\) be a compact bracket-generating
sub-Riemannian manifold and
\[
 \mathcal L_{\mathcal H}=\sum_{\nu=1}^{k}X_\nu^*X_\nu
\]
a sub-Laplacian associated with a smooth volume.  Its principal
symbol vanishes on the characteristic cone \(\mathcal H^\perp\), so
the elliptic phase-space argument no longer applies.  H\"ormander's
condition nevertheless gives hypoellipticity and a discrete spectrum.
The foundational spectral estimates are due to M\'etivier
\cite{Met76} and Menikoff--Sj\"ostrand
\cite{MS78}.  Ponge \cite{Pon08} developed the
Heisenberg calculus as a geometric framework for contact and CR
operators, and Dave and Haller \cite{DH20} proved complete heat
kernel expansions for positive Rockland operators on filtered
manifolds.

In the equiregular case the leading exponent is \(Q/2\), where \(Q\)
is the homogeneous dimension.  Colin de Verdi\`ere, Hillairet, and
Tr\'elat \cite[Theorem~4.1]{VHT22} identify the leading coefficient
through the heat kernel of
the nilpotent approximation and obtain
\[
 \operatorname{Tr}(\e^{-t\mathcal L_{\mathcal H}})
 \sim a_0t^{-Q/2},
 \qquad
 N_{\mathcal H}(\lambda)\sim
 \frac{a_0}{\Gamma(Q/2+1)}\lambda^{Q/2}.
\]
Their local and microlocal Weyl laws
show that the high-energy mass is governed by the nilpotentized
geometry and may concentrate on characteristic directions.  In the
three-dimensional contact case, Colin de Verdi\`ere, Hillairet, and
Tr\'elat \cite{VHT18} proved that quantum limits supported on the
characteristic cone are invariant under the lifted Reeb flow and that
ergodicity of that flow yields quantum ergodicity.  These results
describe the leading density and
averaged microlocal distribution; they do not control the pointwise
oscillation of the unsmoothed counting function.

At the level of graded Lie groups, the nilpotent tangent at each point
of a contact sub-Riemannian manifold is isomorphic to a Heisenberg
group.  The horizontal metric carried by this tangent group, however,
need not reduce to the standard isotropic metric used in this paper
when \(d\geq2\).  Compact quotients equipped with the standard
sub-Laplacian are therefore explicitly diagonalizable model cases,
rather than a universal metric normal form for all contact
sub-Laplacians.  They
nevertheless provide a natural setting in which to determine how
global topology and arithmetic affect the local hypoelliptic Weyl
law.  Their topological dimension is \(2d+1\), whereas
\[
 Q=2d+2,\qquad N_M(\lambda)\asymp\lambda^{Q/2}
 =\lambda^{d+1}.
\]
With \(\mu=\sqrt\lambda\), Strichartz's conjecture reads
\[
 N_M(\mu^2)
 =A_d\operatorname{vol}(M)\mu^Q+O_M(\mu^{Q-2}).
\]
The scale \(\mu^{Q-2}=\lambda^d\) is the first homogeneous order below
the leading order \(\mu^Q\) and is also the size of the
Abelian spectral contribution.  Our theorem shows that the
non-Abelian remainder contains the additional unbounded factor
\(\log\log\mu\).  Thus the homogeneous dimension and the local
nilpotent model determine the leading term but not the sharp
remainder.

This model can be analyzed directly because nilpotentization is exact
globally, characteristic directions are quantized by central Fourier
modes, and the resulting Landau levels have arithmetic
multiplicities.  The exact representation spectrum contains global
information that is absent from heat and Ces\`aro asymptotics.
In the terminology of Colin de Verdi\`ere, Hillairet, and Tr\'elat
\cite{VHT18}, this model gives an explicit quantization of the
characteristic component, although its Reeb flow is periodic rather
than ergodic.

Results for more general nilpotent models also distinguish local
smoothed asymptotics from global spectral data.  Bauer,
Furutani, Iwasaki, and Laaroussi \cite{BFIL} diagonalize sub-Laplacians
on pseudo
\(H\)-type nilmanifolds by central Fourier modes and construct large
families that are sub-Laplacian isospectral but mutually
non-homeomorphic; they also give examples in different dimensions with
heat traces agreeing modulo \(O(t^\infty)\).  Fischer \cite{Fis22}
shows that, on compact graded nilmanifolds, local spectral multiplier
asymptotics for positive Rockland operators can contain only one
nontrivial term.  In both settings, smoothed local data may have a
particularly simple form while the unsmoothed counting function
retains global information.  The two-sided \(\log\log\lambda\)
fluctuation proved here makes this distinction explicit for the basic
contact model and leads to an analogous arithmetic remainder problem
on compact two-step nilmanifolds.

For non-equiregular structures, the pointwise homogeneous dimension
\(\nu(x)\) need not be constant.  Hua Chen and Hong-Ge Chen
\cite[Theorems~1.1--1.3 and Remark~1.1]{CC19} set
\(\widetilde\nu=\max_{x\in X}\nu(x)\) and
\(H=\{x\in X:\nu(x)=\widetilde\nu\}\).  For a self-adjoint sum of
squares on a compact manifold without boundary, they prove a diagonal
heat kernel bound of order \(t^{-\widetilde\nu/2}\), the eigenvalue
lower bound \(\lambda_k\gtrsim k^{2/\widetilde\nu}\), and the counting
asymptotic
\[
 \lim_{\lambda\to\infty}
 \lambda^{-\widetilde\nu/2}N(\lambda)
 =\frac{1}{\Gamma(\widetilde\nu/2+1)}
   \int_H c_0(x)\,dx.
\]
In particular, when \(H\) has positive measure, this gives
\(\lambda_k\sim C k^{2/\widetilde\nu}\).  Hua Chen, Hong-Ge Chen,
and Jin-Ning Li \cite[Theorems~1.1 and~1.2]{CCL22} complement these
results with upper bounds for the closed eigenvalue problem.  Their
general inequality is expressed in terms of the volume of subunit
balls; under the same hypothesis \(\lvert H\rvert>0\), they obtain
\(\lambda_k\lesssim k^{2/\widetilde\nu}\) by the Rayleigh--Ritz
principle and heat kernel estimates.  For the equiregular Heisenberg
manifolds considered here, \(\widetilde\nu=Q=2d+2\) and \(H=X\), so
these general results recover the polynomial eigenvalue scale
\(\lambda_k\asymp k^{1/(d+1)}\).  They do not give a remainder for the
unsmoothed counting function.  The present result addresses that
question and shows that the term of order \(\lambda^d\) contains a
two-sided arithmetic oscillation.

Fan, Kim, and Zeytuncu \cite{FKZ22} obtained the leading Weyl law for
the larger family \(\mathcal L_\alpha\), including the Kohn Laplacians,
by Tauberian methods.  Their result determines the principal
coefficient but does not estimate the unsmoothed remainder considered
here.  The elliptic Laplace--Beltrami problem on a Riemannian
Heisenberg manifold likewise has different eigenvalue scales and a
different lattice point problem.

The paper is organized as follows.  In Section~\ref{spre}, we fix the
spectral normalization and translate Strichartz's reduction into the
notation used here.  In Section~\ref{ptup}, we prove the arithmetic
upper bound in Theorem~\ref{tup} and deduce Theorem~\ref{sup}.  In
Section~\ref{ptom}, we prove the two-sided arithmetic estimates in
Theorem~\ref{tom} and deduce Theorem~\ref{som}.

\section{Preliminaries}
\label{spre}

We fix the spectral normalization required to reduce the counting
remainder \(R_M(\lambda)\) to the sawtooth sum \(S_{d/2}(T)\) in
\eqref{e07}.  The reduction in \eqref{e14} is the only spectral result
used in Sections~\ref{ptup} and~\ref{ptom}.  Folland
\cite[Theorem~3.2]{Fol04} gives the non-Abelian
eigenvalues
\begin{equation}\label{e10}
 \frac{\pi |n|}{2c}(d+2j)
 =\frac{\pi |n|}{c}\left(j+\frac d2\right),
 \qquad j\in\Z_{\geq0},\quad n\in\Z\setminus\{0\},
\end{equation}
with multiplicity
\begin{equation}\label{e11}
 L|n|^d\binom{j+d-1}{d-1}.
\end{equation}
The remaining Abelian eigenvalues form the spectrum of a flat torus
of dimension \(2d\).  In particular, all constants in this section
use the volume normalization \eqref{e03}; we do not rescale the
central variable after the eigenvalue formula \eqref{e10}.

Set
\begin{equation}\label{e12a}
 T=\frac{c\lambda}{\pi},\qquad
 r=\frac d2,\qquad
 a_d(j)=\binom{j+d-1}{d-1}.
\end{equation}
By the eigenvalue and multiplicity formulas \eqref{e10}--\eqref{e11},
the non-Abelian counting function \(N_{\mathrm{na}}(\lambda)\) is
exactly
\begin{equation}\label{e12}
N_{\mathrm{na}}(\lambda)
=2L\sum_{\substack{j\geq0\\j+r\leq T}}
 a_d(j)\sum_{1\leq n\leq T/(j+r)}n^d.
\end{equation}

The following proposition is Strichartz's spectral reduction, written
in the normalization fixed above.
\begin{proposition}[Strichartz]
\label{pred}
Let
\[
 R_M(\lambda)
 =N_M(\lambda)-A_d\operatorname{vol}(M)\lambda^{d+1}.
\]
Then, as \(\lambda\to\infty\),
\begin{equation}\label{e14}
 R_M(\lambda)
 =-\frac{2L}{(d-1)!}
 \left(\frac{c\lambda}{\pi}\right)^d
 S_{d/2}\left(\frac{c\lambda}{\pi}\right)
+O_M(\lambda^d).
\end{equation}
Moreover,
\begin{equation}\label{e15}
\begin{split}
A_d
&=\frac{2}{(d+1)\pi^{d+1}}
 \sum_{j=0}^{\infty}
 \frac{\binom{j+d-1}{d-1}}{(j+d/2)^{d+1}}\\
&=\frac{2}{\pi^{d+1}\Gamma(d+2)}
 \int_{-\infty}^{\infty}
 \left(\frac{x}{\sinh x}\right)^d\,dx .
\end{split}
\end{equation}
\end{proposition}

\begin{proof}
Strichartz \cite[p.~2455]{Str16} proves the reduction for
\(R_M(\lambda)\) in \eqref{e14}.  We only
give the translation to the present notation and normalization.  His
spectral parameter \(t\) is our \(\lambda\).  After the lattice is put
in Folland's normal form, the central period is \(c\),
the multiplicity of the representation with central frequency \(n\)
is \(L|n|^d\), and the eigenvalue and multiplicity formulas are
\eqref{e10}--\eqref{e11}.  Thus the non-Abelian counting identity
\eqref{e12} is the starting sum in Strichartz's calculation.
Consequently, the variables in the
calculation on that page become \(T=c\lambda/\pi\), \(r=d/2\), and
\(a_d(j)=\binom{j+d-1}{d-1}\), as specified in \eqref{e12a}.  In this
notation, the formula obtained there reads
\begin{align}
 R_M(\lambda)={}&-2LT^d
 \sum_{j+r\leq T}\frac{a_d(j)}{(j+r)^d}
 \psi\left(\frac{T}{j+r}\right)
+O_M(T^d).
\label{e16}
\end{align}
For \(j\geq1\), the weight in the sum in \eqref{e16} satisfies
\begin{equation}\label{e17}
 \frac{a_d(j)}{(j+d/2)^d}
 =\frac{1}{(d-1)!\,j}+O_d(j^{-2}).
\end{equation}
Since \(|\psi|\leq1/2\), the \(O_d(j^{-2})\) terms in \eqref{e17}
are summable.  The term \(j=0\) and the change from
\(j+d/2\leq T\) to \(1\leq j\leq T\) contribute \(O_d(1)\).
Therefore, with the sawtooth sum \(S_r(T)\) defined in \eqref{e07},
\begin{equation}\label{e18}
 \sum_{j+d/2\leq T}\frac{a_d(j)}{(j+d/2)^d}
 \psi\left(\frac{T}{j+d/2}\right)
 =\frac{1}{(d-1)!}S_{d/2}(T)+O_d(1).
\end{equation}
Substituting \eqref{e18} into \eqref{e16} and using
\(T=c\lambda/\pi\) gives the reduction for \(R_M(\lambda)\) in
\eqref{e14}.  The main term
in the same calculation, together with the volume
identity \eqref{e03}, gives the first expression for \(A_d\) in
\eqref{e15}.  Finally, the generating function
\(\sum_{j\geq0}a_d(j)z^j=(1-z)^{-d}\), the Mellin formula, and the
change of variables \(t=2x\) give the second identity for \(A_d\)
in \eqref{e15}.
\end{proof}

\section{Proof of Theorem~\ref{sup}}
\label{ptup}

By the reduction for \(R_M(\lambda)\) in \eqref{e14},
Theorem~\ref{sup} follows from
an upper estimate for the sawtooth sum \(S_r(T)\) in \eqref{e07}.
We prove the following more general result for every fixed nonnegative rational
shift.

\begin{theorem}\label{tup}
For every fixed \(r\in\mathbb Q_{\geq0}\),
\begin{equation}\label{e08}
 S_r(T)=O_r\!\left((\log T)^{2/3}\right)
 \qquad(T\geq2).
\end{equation}
\end{theorem}

The proof of Theorem~\ref{tup} uses Walfisz's estimate on arithmetic
progressions.  We begin with an identity that relates \(S_0(T)\) in
\eqref{e07} to the summatory reciprocal divisor function.

\begin{lemma}
Let \(N=\lfloor T\rfloor\).  Then
\begin{equation}\label{e35}
 S_0(T)=
 T\sum_{m\leq N}\frac1{m^2}
 -\sum_{n\leq T}\sigma_{-1}(n)
 -\frac12H_N,
 \qquad H_N=\sum_{m\leq N}\frac1m,
 \quad
 \sigma_{-1}(n)=\sum_{a\mid n}\frac1a.
\end{equation}
\end{lemma}

\begin{proof}
Expand the sawtooth in the definition of \(S_0(T)\) in \eqref{e07}.
It remains to identify the term containing the floor function:
\[
 \sum_{m\leq N}\frac1m\left\lfloor\frac Tm\right\rfloor
 =\sum_{mk\leq T}\frac1m
 =\sum_{n\leq T}\sum_{m\mid n}\frac1m.
\]
This is \(\sum_{n\leq T}\sigma_{-1}(n)\).
\end{proof}

The required classical estimate is due to Walfisz.

\begin{lemma}[Walfisz]\label{lwal}
Fix \(N\in\N\) and \(1\leq a,b\leq N\).  With the convention for
\(\psi\) used in the definition of \(S_r(T)\) in \eqref{e07},
\begin{equation}\label{e36}
 T(x;a,b,N)
 :=\sum_{\substack{m\leq Nx\\m\equiv a\;(\mathrm{mod}\,N)}}
 \frac1m\psi\left(\frac{x}{m}-\frac bN\right)
 \ll_N(\log x)^{2/3}\qquad(x\geq2).
\end{equation}
In particular, \(S_0(T)=O((\log T)^{2/3})\).
\end{lemma}

\begin{proof}
Walfisz \cite[p.~272]{Wal60} takes \(a\) and \(b\) to be
integers between \(1\) and \(N\), and a prime on the summation sign
denotes the congruence \(m\equiv a\pmod N\).  With these conventions,
the asserted estimate is Hilfssatz~5, equations (120)--(121), on
pp.~275--276 of Walfisz \cite{Wal60}:
\[
 T(x;a,b,N)
 =\sum_{\substack{m\leq Nx\\m\equiv a\;(\mathrm{mod}\,N)}}
 \frac1m\psi\left(\frac{x}{m}-\frac bN\right)
 \ll_N(\log x)^{2/3}.
\]
His definition (98) is
\(\psi(u)=u-\lfloor u\rfloor-\tfrac12\), including the value
\(-1/2\) at integers, so there is no endpoint convention
conversion.  The zero residue in the translation is represented by
\(b=N\), because \(\psi(u-1)=\psi(u)\).  Taking \(N=a=b=1\) gives
\(S_0(T)\ll(\log T)^{2/3}\) for all sufficiently large \(T\).
Enlarging the absolute constant covers \(2\leq T\leq T_0\).
\end{proof}

\begin{lemma}
\label{lrat}
Let \(r=A/B\geq0\), where \(A\in\Z_{\geq0}\), \(B\in\N\), and
\((A,B)=1\).  Then
\begin{equation}\label{e37}
 S_{A/B}(T)=O_{A,B}\!\left((\log T)^{2/3}\right)
 \qquad(T\geq2).
\end{equation}
\end{lemma}

\begin{proof}
Put \(N=\lfloor T\rfloor\) and make the change of variables
\(j=Bm+A\).  Since
\[
 \frac{B}{j-A}=\frac{B}{j}
 +O_{A,B}\!\left(\frac1{j^2}\right)
 \qquad(j\geq B+A),
\]
absolute summation of the error gives
\begin{equation}\label{e38}
 S_{A/B}(T)
 =
 B\!\!\sum_{\substack{B+A\leq j\leq BN+A\\
                       j\equiv A\;(\mathrm{mod}\,B)}}
 \frac1j\psi\left(\frac{BT}{j}\right)+O_{A,B}(1).
\end{equation}
Write \(x=BT\), and let \(a\in\{1,\ldots,B\}\) represent the residue
class of \(A\) modulo \(B\).  The progression sum
\(T(x;a,B,B)\), defined in \eqref{e36}, is
\[
 T(x;a,B,B)
 =\sum_{\substack{j\leq Bx\\j\equiv a\;(\mathrm{mod}\,B)}}
 \frac1j\psi\left(\frac{x}{j}-1\right)
 =\sum_{\substack{j\leq Bx\\j\equiv a\;(\mathrm{mod}\,B)}}
 \frac1j\psi\left(\frac{x}{j}\right)
 \ll_B(\log x)^{2/3}.
\]
The finitely many initial terms omitted from the expression for
\(S_{A/B}(T)\) in \eqref{e38} contribute \(O_{A,B}(1)\).
Changing the upper endpoint of the progression sum for
\(S_{A/B}(T)\) in \eqref{e38} to \(Bx=B^2T\) introduces an additional
term of order \(O_{A,B}(1)\).  Indeed,
\(BN+A=x+A-B\{T\}=x+O_{A,B}(1)\).  Apart from the bounded number
of terms between these two endpoints, the added range satisfies
\(x<j\leq Bx\), and there
\(\psi(x/j)=x/j-1/2\); hence
\[
 \sum_{\substack{x<j\leq Bx\\j\equiv a\;(\mathrm{mod}\,B)}}
 \frac1j\left|\frac{x}{j}-\frac12\right|
 \ll_B 1.
\]
Thus the decomposition for \(S_{A/B}(T)\) in \eqref{e38} becomes
\[
 S_{A/B}(T)=B\,T(BT;a,B,B)+O_{A,B}(1).
\]
The estimate for \(T(BT;a,B,B)\) in \eqref{e36} now proves the bound for
\(S_{A/B}(T)\) in \eqref{e37}.
\end{proof}

\begin{remark}
Lemma~\ref{lwal} is the weighted sawtooth estimate used in
Walfisz's four-dimensional ellipsoid remainder, and it is the form
needed here.  The identity for \(S_0(T)\) in \eqref{e35} gives the
equivalent reciprocal divisor formulation.  The progression form
also shows that the exponent \(2/3\) is uniform in \(T\) for every
fixed rational shift.  The direct argument in Section~\ref{ptom}
treats the two signs for rational shifts directly.
\end{remark}

\begin{proof}[Proof of Theorem~\ref{tup}]
Write \(r=A/B\) in lowest terms, with \(A\in\Z_{\geq0}\) and
\(B\in\N\).  The estimate for \(S_{A/B}(T)\) in \eqref{e37}, proved
in Lemma~\ref{lrat}, proves the estimate for \(S_r(T)\) in
\eqref{e08}.
\end{proof}

\begin{proof}[Proof of Theorem~\ref{sup}]
Substitute the estimate for \(S_{d/2}(c\lambda/\pi)\) from
Theorem~\ref{tup}, namely \eqref{e08}, into the reduction for
\(R_M(\lambda)\) in \eqref{e14}.  Since
\(\log(2+c\lambda/\pi)\asymp_M\log(2+\lambda)\), the asserted
spectral upper bound \eqref{e06} follows.
\end{proof}

\section{Proof of Theorem~\ref{som}}
\label{ptom}

The reduction for \(R_M(\lambda)\) in \eqref{e14} also converts
two-sided estimates
for \(S_{d/2}(T)\) in \eqref{e07} into lower bounds for the spectral
remainder.  We prove the following arithmetic theorem.

\begin{theorem}\label{tom}
For every fixed \(r\in\mathbb Q_{\geq0}\),
\begin{equation}\label{e09}
 \limsup_{T\to\infty}
 \frac{S_r(T)}{\log\log T}\geq\frac12,
 \qquad
 \liminf_{T\to\infty}
 \frac{S_r(T)}{\log\log T}\leq-\frac12.
\end{equation}
If \(2r\in\Z\), then the stronger estimates hold:
\begin{equation}\label{e09b}
 \limsup_{T\to\infty}
 \frac{S_r(T)}{\log\log T}\geq\frac{\e^\gamma}{2},
 \qquad
 \liminf_{T\to\infty}
 \frac{S_r(T)}{\log\log T}\leq-\frac{\e^\gamma}{2}.
\end{equation}
\end{theorem}

The stronger bounds for \(S_r(T)\) in \eqref{e09b} apply to the
spectral shift \(r=d/2\).  We first establish these bounds from
P\'etermann's theorem.  A separate argument then proves the two-sided
bounds for \(S_r(T)\) in \eqref{e09} at every nonnegative rational
shift.

\subsection{The spectral shifts and P\'etermann's theorem}

We use the following result of P\'etermann \cite{Pet87,Pet88}.

\begin{lemma}[P\'etermann]\label{lpet}
Let
\[
 E_{-1}(x)=\sum_{n\leq x}\sigma_{-1}(n)
 -\zeta(2)x+\frac12\log x.
\]
Then
\begin{equation}\label{ep1}
 \limsup_{x\to\infty}\frac{E_{-1}(x)}{\log\log x}
 \geq\frac{\e^\gamma}{2},
 \qquad
 \liminf_{x\to\infty}\frac{E_{-1}(x)}{\log\log x}
 \leq-\frac{\e^\gamma}{2}.
\end{equation}
\end{lemma}

The next lemma transfers these estimates to the integer and
half-integer shifts that occur in the Heisenberg spectrum.

\begin{lemma}\label{lshf}
For every fixed integer \(k\geq0\),
\begin{align}
 S_k(T)&=S_0(T)+O_k(1),\label{ep2}\\
 S_{k+1/2}(T)&=2S_0(2T)-S_0(T)+O_k(1).\label{ep3}
\end{align}
Consequently the stronger bounds for \(S_r(T)\) in \eqref{e09b} hold
for every \(r\geq0\) such that \(2r\in\Z\).
\end{lemma}

\begin{proof}
Let \(N=\lfloor T\rfloor\).  From the identity for \(S_0(T)\) in
\eqref{e35}, together with
\[
 T\sum_{m\leq N}\frac1{m^2}=\zeta(2)T+O(1),
 \qquad H_N=\log T+O(1),
\]
we obtain
\begin{equation}\label{ep4}
 S_0(T)=-E_{-1}(T)+O(1).
\end{equation}
Lemma~\ref{lpet} and the relation \eqref{ep4} therefore give both
inequalities for \(S_0(T)\) in \eqref{e09b}.

For an integer shift, set \(j=m+k\).  Since
\[
 \frac1{j-k}=\frac1j+O_k(j^{-2})
 \qquad(j\geq k+1),
\]
absolute summation of the error, followed by the addition or removal
of the finitely many end terms, gives
\(S_k(T)=S_0(T)+O_k(1)\) in \eqref{ep2}.

For a half-integer shift, put \(q=2m+2k+1\).  Replacing
\(2/(q-2k-1)\) by \(2/q\) and adjusting the two endpoints introduce
terms of order \(O_k(1)\).  Hence
\begin{align*}
 S_{k+1/2}(T)
 &=2\sum_{\substack{q\leq2T\\q\ \mathrm{odd}}}
   \frac1q\psi\left(\frac{2T}{q}\right)+O_k(1)\\
 &=2S_0(2T)-S_0(T)+O_k(1),
\end{align*}
which proves
\(S_{k+1/2}(T)=2S_0(2T)-S_0(T)+O_k(1)\) in \eqref{ep3}.

To verify that this transformation preserves the constants, define
\begin{equation}\label{hfun}
 H(T)=2S_0(2T)-S_0(T).
\end{equation}
Iteration of the identity for \(H(T)\) in \eqref{hfun} gives
\begin{equation}\label{ep5}
 S_0(x)=2^{-K}S_0(x/2^K)
 +\sum_{j=1}^{K}2^{-j}H(x/2^j).
\end{equation}
Put \(C=\e^\gamma/2\).  If
\(\limsup H(T)/\log\log T<C\), choose \(c<C\) and \(T_0>\e\) so
that \(H(T)\leq c\log\log T\) for \(T\geq T_0\).  For each large
\(x\), choose \(K\) so that \(x/2^K\in[T_0,2T_0)\).  Since all
coefficients in \eqref{ep5} are positive and
\(\log\log(x/2^j)\leq\log\log x\), we get
\[
 S_0(x)\leq c\log\log x+O_{T_0}(1),
\]
contrary to P\'etermann's bounds for \(E_{-1}(T)\) in \eqref{ep1}
and the relation between \(S_0(T)\) and \(E_{-1}(T)\) in
\eqref{ep4}.  Thus
\(\limsup H(T)/\log\log T\geq C\).  Applying the same argument to
\(-S_0\) and \(-H\) proves the corresponding liminf inequality for
\(H(T)\).  The identities for \(S_k(T)\) and \(S_{k+1/2}(T)\) in
\eqref{ep2} and \eqref{ep3} now prove the stronger bounds
for \(S_r(T)\) in \eqref{e09b}.
\end{proof}

\subsection{Centered sawtooth and hyperbola truncation}

We turn to the bounds for \(S_r(T)\) in \eqref{e09} for every
nonnegative rational shift.  We begin by separating the midpoint of
each jump from its two endpoint values.  Define the centered sawtooth
\begin{equation}\label{e39}
 \psic(u)=
 \begin{cases}
  \{u\}-\dfrac12,&u\notin\Z,\\[1mm]
  0,&u\in\Z,
 \end{cases}
\end{equation}
and, for \(x\geq1\),
\begin{equation}\label{e40}
 F_r(x)=\sum_{1\leq m\leq x}\frac1m
 \psic\left(\frac{x}{m+r}\right).
\end{equation}
The value \(0\) in the definition of \(\psic\) in \eqref{e39} gives
exact mean zero on every rational orbit.

\begin{lemma}\label{lhyp}
For every fixed \(r\geq0\), uniformly for \(x\geq2\),
\begin{equation}\label{e41}
 F_r(x)=
 \sum_{m\leq\sqrt{x}}\frac1m
 \psic\left(\frac{x}{m+r}\right)+O_r(1).
\end{equation}
\end{lemma}

\begin{proof}
Since \(|\psic|\leq1/2\),
\[
 \sum_{m>\sqrt{x}}
 \left|\frac1m-\frac1{m+r}\right|
 \left|\psic\left(\frac{x}{m+r}\right)\right|
 \ll_r\sum_{m>\sqrt{x}}\frac1{m^2}\ll_r x^{-1/2}.
\]
It suffices to estimate
\[
 \sum_{\sqrt{x}<m\leq x}\frac1{m+r}
 \psic\left(\frac{x}{m+r}\right).
\]
Write \(u=m+r\), so that \(u\) runs through a translate of the
integer lattice.  On
\[
 \frac{x}{k+1}<u<\frac{x}{k}\qquad(k\geq1)
\]
the summand equals
\[
 f_k(u)=\frac{x}{u^2}-\frac{k+1/2}{u}.
\]
If an endpoint belongs to the translated lattice, its contribution
is zero by the definition of \(\psic\) in \eqref{e39}; hence the
lattice sum is taken over the open interval.

For a continuously differentiable function \(f\) on \([a,b]\), a
sum over a lattice with spacing one satisfies
\begin{equation}\label{e42}
 \sum_{\substack{u\in r+\Z\\a<u<b}}f(u)
 =\int_a^b f(v)\,dv+
 O\left(|f(a)|+|f(b)|+\int_a^b|f'(v)|\,dv\right).
\end{equation}
This form of the Euler summation formula has an absolute implied
constant.  To verify it, partition \([a,b]\) into unit cells
of the translated lattice.  On every complete cell, the difference
between the value at its lattice point and the integral is bounded
by the integral of \(|f'|\) over that cell; the two incomplete end
cells account for \(|f(a)|+|f(b)|\).

For \(a=x/(k+1)\) and \(b=x/k\), direct calculation gives
\begin{equation}\label{e43}
 f_k(a)=\frac{k+1}{2x},
 \qquad
 f_k(b)=-\frac{k}{2x}.
\end{equation}
Furthermore,
\[
 f_k'(v)=-\frac{2x}{v^3}+\frac{k+1/2}{v^2},
\]
and hence
\begin{align*}
 \int_a^b|f_k'(v)|\,dv
 &\leq
 2x\int_a^b\frac{dv}{v^3}
 +\left(k+\frac12\right)\int_a^b\frac{dv}{v^2}\\
 &=\frac{2k+1}{x}+\frac{k+1/2}{x}
 \ll\frac{k+1}{x}.
\end{align*}
Together with the endpoint identities for \(f_k(a)\) and \(f_k(b)\)
in \eqref{e43}, this gives
\[
 |f_k(a)|+|f_k(b)|+\int_a^b|f_k'(v)|\,dv
 \ll \frac{k+1}{x}.
\]
Moreover,
\begin{align}
 \int_{x/(k+1)}^{x/k}f_k(v)\,dv
 &=1-\left(k+\frac12\right)\log\left(1+\frac1k\right)
 \ll\frac1{k^2}.
\label{e44}
\end{align}
The last bound follows, including \(k=1\), from Taylor's formula
with remainder for \(\log(1+z)\) on \(0\leq z\leq1\).
Only \(1\leq k\leq\sqrt{x}\) occur when \(u>\sqrt{x}+r\).  Summing the
estimate for the block integral of \(f_k\) in \eqref{e44} and the
corresponding error terms
gives
\[
 \sum_{k\ll\sqrt{x}}\left(\frac1{k^2}+\frac{k+1}{x}\right)\ll1.
\]
There is at most one incomplete block at the lower endpoint
\(u=\sqrt{x}+r\); applying the Euler summation formula for \(f\) in
\eqref{e42} on the corresponding subinterval gives \(O_r(1)\).  On the
remaining interval \(x<u\leq x+r\), for which \(k=0\), there are \(O_r(1)\)
lattice points and every summand is \(O(x^{-1})\).  This piece is
also \(O_r(1)\).
This proves the truncation formula for \(F_r(x)\) in \eqref{e41}.
\end{proof}

\subsection{A bounded mean on a progression}

For rational \(r\), the required mean bound follows from exact
periodicity on progressions.  The length \(Q^3\) below is chosen for
convenience; any \(N\) with
\(N/Q\to\infty\) would suffice.

\begin{lemma}
Fix \(r=A/B\in\mathbb Q_{\geq0}\), where
\(A\in\Z_{\geq0}\), \(B\in\N\), and \((A,B)=1\).  For every
integer \(Q\geq2\),
with \(N=Q^3\),
\begin{equation}\label{e45}
 \frac1N\sum_{1\leq n\leq N}F_r(nQ)=O_r(1),
\end{equation}
where the implied constant is independent of \(Q\).
\end{lemma}

\begin{proof}
By Lemma~\ref{lhyp}, specifically the truncation formula for \(F_r(x)\)
in \eqref{e41},
\begin{align}
 \sum_{n\leq N}F_r(nQ)
 &=
 \sum_{n\leq N}
 \sum_{m\leq\sqrt{nQ}}\frac1m
 \psic\left(\frac{BnQ}{Bm+A}\right)+O_r(N).
\label{e46}
\end{align}
Put \(M=\lfloor\sqrt{NQ}\rfloor\).  After reversing the order of
summation, the inner \(n\)-sum associated with a fixed \(m\leq M\)
is taken over the interval
\[
 \max\left\{1,\left\lceil\frac{m^2}{Q}\right\rceil\right\}
 \leq n\leq N.
\]

Let
\[
 g_m=(BQ,Bm+A),\qquad q_m=\frac{Bm+A}{g_m}.
\]
The sequence
\[
 n\longmapsto
 \psic\left(\frac{BnQ}{Bm+A}\right)
\]
has period \(q_m\).  Since \(BQ/g_m\) is coprime to \(q_m\), one
complete period is a permutation of
\(\psic(j/q_m)\), \(0\leq j<q_m\).  Hence
\begin{equation}\label{e47}
 \sum_{j=0}^{q_m-1}\psic\left(\frac{j}{q_m}\right)
 =\sum_{j=1}^{q_m-1}\left(\frac{j}{q_m}-\frac12\right)=0.
\end{equation}
Decompose any interval of consecutive \(n\)'s into complete periods
and at most two incomplete end pieces.  The mean-zero identity for
\(\psic\) in \eqref{e47} shows
that the complete periods contribute zero, while the two end pieces
contain fewer than \(2q_m\) terms of absolute value at most \(1/2\).
The sum over the interval is \(O(q_m)\).  Its contribution to the
double sum for \(F_r(nQ)\) in \eqref{e46}, including
the factor \(1/m\), is
\[
 O\left(\frac{q_m}{m}\right)
 =O\left(\frac{Bm+A}{m g_m}\right)=O_r(1).
\]
Summing this estimate for \(m\leq M\) gives
\[
 \sum_{n\leq N}F_r(nQ)=O_r(N+M).
\]
Since \(M\leq\sqrt{NQ}=Q^2\) and \(N=Q^3\), division by \(N\)
proves the mean estimate for \(F_r(nQ)\) in \eqref{e45}.
\end{proof}

\begin{corollary}\label{cmid}
For every fixed \(r\in\mathbb Q_{\geq0}\), there is a constant
\(K_r>0\) such that, for every \(Q\geq2\), one
can find \(n_+,n_-\in[1,Q^3]\cap\Z\) satisfying
\begin{equation}\label{e48}
 F_r(n_+Q)\geq-K_r,
 \qquad
 F_r(n_-Q)\leq K_r.
\end{equation}
\end{corollary}

\begin{proof}
The mean estimate for \(F_r(nQ)\) in \eqref{e45} implies that the
average of \(F_r(nQ)\), \(1\leq n\leq Q^3\), has absolute value at most a
constant \(K_r\) independent of \(Q\).  At least one term is no
smaller than this average, and at least one term is no larger.  These
two terms satisfy the inequalities for \(F_r(n_+Q)\) and
\(F_r(n_-Q)\) in \eqref{e48}.
\end{proof}

\subsection{Simultaneous jumps}

For every fixed \(r\geq0\), the function \(S_r\) has only finitely
many discontinuities in each bounded interval.  Define its left limit by
\begin{equation}\label{e49a}
 S_r(T^-)=\lim_{U\uparrow T}S_r(U).
\end{equation}
The limit exists because near a fixed \(T\) only finitely many
summands occur, and each summand has a left limit.  If
\(r=A/B\in\mathbb Q_{\geq0}\), then
\begin{equation}\label{e49}
 \operatorname{Disc}(S_r)\subset B^{-1}\mathbb Z.
\end{equation}
Indeed, a sawtooth discontinuity belonging to the \(m\)-th summand
occurs only at
\(T=k(m+r)=k(Bm+A)/B\), while a change of the summation endpoint
occurs at an integer.  Both types of points belong to
\(B^{-1}\mathbb Z\).

For \(T\in\N\), let
\begin{equation}\label{e49b}
 \mathcal J_r(T)=
 \left\{m<T:\frac{T}{m+r}\in\Z\right\},
 \qquad
 J_r(T)=\sum_{m\in\mathcal J_r(T)}\frac1m.
\end{equation}

\begin{lemma}
For every \(r\geq0\) and \(T\in\N\), with \(F_r\) and \(J_r\)
defined in \eqref{e40} and \eqref{e49b}, respectively,
\begin{align}
 S_r(T)
 &=F_r(T)-\frac12J_r(T)
   -\frac{\boldsymbol 1_{\{r=0\}}}{2T},
\label{e50}\\
 S_r(T^-)
 &=F_r(T)+\frac12J_r(T)
   -\frac1T\psic\left(\frac{T}{T+r}\right).
\label{e51}
\end{align}
In particular,
\begin{align}
 S_r(T)&\leq F_r(T)-\frac12J_r(T),
\label{e52}\\
 S_r(T^-)&\geq F_r(T)+\frac12J_r(T)-\frac1{2T}.
\label{e53}
\end{align}
\end{lemma}

\begin{proof}
For \(m<T\), the summands in \(S_r(T)\) and \(F_r(T)\) agree unless
\(T/(m+r)\) is an integer.  At such an index their difference is
\(-1/(2m)\), while the corresponding left limit differs from
\(F_r(T)\) by \(+1/(2m)\).  The term \(m=T\) is present in both
sums at the right endpoint.  It contributes an additional
\(-1/(2T)\) precisely when \(r=0\), proving the identity for
\(S_r(T)\) in \eqref{e50}.  It is absent from the left limit, whereas
its contribution to \(F_r(T)\) is
\(T^{-1}\psic(T/(T+r))\), which proves the identity for
\(S_r(T^-)\) in \eqref{e51}.  The inequalities for \(S_r(T)\) and
\(S_r(T^-)\) in \eqref{e52} and \eqref{e53} follow from
\(|\psic|\leq1/2\).
\end{proof}

\begin{proof}[Proof of Theorem~\ref{tom}]
Write \(r=A/B\) in lowest terms, where
\(A\in\Z_{\geq0}\) and \(B\in\N\).
Let \(y>A+B\) tend to infinity through the integers, and put
\begin{equation}\label{qdef}
 Q=Q_y=\operatorname{lcm}(1,2,\ldots,y),\qquad N=Q^3.
\end{equation}
If \(T=nQ\), \(1\leq n\leq N\), and
\begin{equation}\label{ldef}
 1\leq m\leq L_y:=\left\lfloor\frac{y-A}{B}\right\rfloor,
\end{equation}
then \(Bm+A\leq y\), hence \(Bm+A\mid Q\), and
\[
 \frac{T}{m+r}=\frac{BnQ}{Bm+A}\in\Z.
\]
For \(y\geq3\), the integer \(Q\) is divisible by
\(\operatorname{lcm}(y-1,y)=y(y-1)\).  Thus the quantity \(L_y\)
in \eqref{ldef} satisfies \(L_y\leq y<Q\leq T\), and every such
\(m\) belongs to \(\mathcal J_r(T)\) as defined in
\eqref{e49b}, including the strict restriction \(m<T\).
Consequently there is a constant \(C_r>0\), independent of
\(y\) and \(n\), such that, uniformly for \(1\leq n\leq N\),
\begin{equation}\label{e54}
 J_r(nQ)\geq H_{L_y}\geq\log y-C_r.
\end{equation}
Here the second inequality follows from
\(L_y=y/B+O_r(1)\) and \(H_n=\log n+O(1)\).

Choose \(n_-\) from Corollary~\ref{cmid} with
\(F_r(n_-Q)\leq K_r\), and set \(T_-=n_-Q\).
By the right endpoint estimate for \(S_r(T)\) in \eqref{e52} and the
lower bound for \(J_r(nQ)\) in \eqref{e54},
\begin{equation}\label{e55}
 S_r(T_-)
 \leq-\frac12\log y+C_r,
\end{equation}
after enlarging \(C_r\) if necessary.
Likewise choose \(n_+\) with
\(F_r(n_+Q)\geq-K_r\), and put \(T_+=n_+Q\).
The left limit estimate for \(S_r(T^-)\) in \eqref{e53} and the lower
bound for \(J_r(nQ)\) in \eqref{e54} give
\begin{equation}\label{e56}
 S_r(T_+^-)
 \geq\frac12\log y-C_r.
\end{equation}
By the description of \(\operatorname{Disc}(S_r)\) in \eqref{e49},
the interval
\((T_+-1/B,T_+)\) contains no discontinuity of \(S_r\).  For
\(y>B\), the definition of \(S_r(T^-)\) in \eqref{e49a} and the
lower bound for \(S_r(T_+^-)\) in \eqref{e56} allow us to
choose
\(U_+\in(T_+-1/y,T_+)\) so that
\begin{equation}\label{e56a}
 S_r(U_+)\geq S_r(T_+^-)-1
 \geq\frac12\log y-C_r-1.
\end{equation}
Put \(U_-=T_-\).

To complete the proof, we compare \(y\) with the selected points.
For \(Q_y\) defined in \eqref{qdef}, the classical Chebyshev bounds
for the second Chebyshev function imply
\[
 \log Q_y=\sum_{p^a\leq y}\log p\asymp y;
\]
see Hardy and Wright \cite[Chapter~22]{HW08}.  For all sufficiently
large \(y\),
\[
 \frac Q2\leq U_\pm\leq Q^4,
\]
because \(T_\pm=n_\pm Q\), \(1\leq n_\pm\leq Q^3\), and
\(U_+\in(T_+-1/y,T_+)\).  Therefore, uniformly for the selected points,
\begin{equation}\label{e57}
 \log\log U_\pm=\log y+O(1).
\end{equation}
The estimates for \(S_r(U_-)\) and \(S_r(U_+)\) in \eqref{e55} and
\eqref{e56a}, together with the comparison for \(\log\log U_\pm\) in
\eqref{e57}, prove the
two-sided bounds \eqref{e09} after
division by the corresponding \(\log\log U_\pm\) and passage to
\(y\to\infty\).  Since \(Q_y\to\infty\), both selected sequences
\(U_-\) and \(U_+\) tend to infinity.  Lemma~\ref{lshf} supplies the
stronger estimates for \(S_r(T)\) in \eqref{e09b} for the integer and
half-integer shifts, completing the proof of Theorem~\ref{tom}.
\end{proof}

\begin{proof}[Proof of Theorem~\ref{som}]
Put \(T=c\lambda/\pi\).  After dividing the spectral reduction for
\(R_M(\lambda)\) in \eqref{e14} by
\(\lambda^d\log\log\lambda\), the \(O_M(\lambda^d)\) term tends to
zero.  The negative sign in the reduction for \(R_M(\lambda)\) in
\eqref{e14}
exchanges the limsup and liminf.  Applying Theorem~\ref{tom} and using
\(\log\log T=\log\log\lambda+o(1)\) gives the inequalities
\eqref{e05}, with
\[
 \frac{2L(c/\pi)^d}{(d-1)!}\cdot\frac{\e^\gamma}{2}
 =\e^\gamma\kappa_M.
\]
The proof is completed.
\end{proof}

\pdfbookmark[1]{Acknowledgements}{acks}
\section*{Acknowledgements}
\noindent S.-C. Mao is partially supported by the China Postdoctoral Science
Foundation (Grant No.~2026M793367).  Y. Zhang has received funding from the European
Research Council (ERC) under the European Union's Horizon 2020 research and
innovation programme (grant agreement GEOSUB, No.~945655).
\medskip

\bigskip

\noindent\textbf{Competing interests.}
The authors declare no competing interests.

\bigskip
\noindent\textbf{Data availability.}
No datasets were generated or analyzed in this study.

\bigskip
\noindent\textbf{Declaration of AI Use.}
During the preparation of this manuscript, the authors used
OpenAI GPT-5.6 for English language editing and stylistic improvements, including grammar, wording, and readability. All mathematical content was developed and verified by the authors, who take full responsibility for the final manuscript.

\pdfbookmark[1]{References}{refs}

\medskip
\mbox{}\\
Sheng-Chen Mao (\textit{Corresponding author})\\
School of Mathematics and Statistics \\
Lanzhou University \\
No. 222 Tianshui South Road \\
Lanzhou 730000, P.R. China \\
\noindent
\begin{tabular}{@{}ll@{}}
{E-Mails:}&{\ttfamily maoshengchen@lzu.edu.cn; maosci@163.com }
\end{tabular}
 
\bigskip\noindent
Ye Zhang  \\
SISSA  \\
via Bonomea, 265 \\
34136 Trieste, Italy \\
\noindent
\begin{tabular}{@{}ll@{}}
{E-Mails:}&{\ttfamily yezhang@sissa.it; zhangye0217@gmail.com }
\end{tabular}

\end{document}